\documentclass{amsart}

\usepackage{ifpdf}
\usepackage{float}

\usepackage{breqn}
\makeatletter
\AtBeginDocument{% <-- add this to solve the eternal "labels changed"
    \catcode`_=12
    \begingroup\lccode`~=`_
    \lowercase{\endgroup\let~}\sb
    \mathcode`_="8000
    \immediate\write\@auxout{\catcode`_=12 }% <-- add this
    \immediate\write\@auxout{\catcode`^=12 }% <-- add this
}
\makeatother

\usepackage{bm}
\usepackage{array}
\usepackage{mathtools}
\usepackage{xargs}
\usepackage{amssymb, amsmath}

\usepackage{graphicx}

\usepackage{enumitem}
\usepackage{soul, color}

\usepackage{listings}
\DeclareMathAlphabet{\mathsfit}{T1}{\sfdefault}{\mddefault}{\sldefault}
\SetMathAlphabet{\mathsfit}{bold}{T1}{\sfdefault}{\bfdefault}{\sldefault}
\newcommand\Qmat{\mathsfit{Q}}

\newcommand\R{\mathbb{R}}

\renewcommand\st{\mathrel{\ooalign{$\,\backepsilon$\cr\lower .7pt\hbox{\kern 1pt$-\,$}}}}

\title{}
\date{}
\theoremstyle{plain}
\newtheorem{theorem}{Theorem}[section]
\newtheorem{lemma}[theorem]{Lemma}

\newtheorem{proposition}[theorem]{Proposition}

\theoremstyle{definition}

\newtheorem{remark}[theorem]{Remark}

\newtheorem{example}[theorem]{Example}

\title[Pythagoras number for polynomials of degree 4 in 5 variables]{On the Pythagoras number for polynomials of degree 4 in 5 variables}

\author{Santiago Laplagne}
\address{Instituto de C\'alculo, FCEN, Universidad de Buenos Aires - Ciudad Universitaria, Pabell\'on I - (C1428EGA) - Buenos Aires, Argentina}
\email{slaplagn@dm.uba.ar}
\begin{document}

\begin{abstract}
We give an example of a polynomial of degree 4 in 5 variables that is the sum of squares of 8 polynomials and cannot be decomposed as the sum of 7 squares. This improves the current existing lower bound of 7 polynomials for the Pythagoras number $p(5,4)$.

% REQUIRED
%\begin{keywords}
%  positive polynomials, sum of squares, Gram spectrahedron, Hilbert function
%\end{keywords}

% REQUIRED
%\begin{AMS}
%  14P10, 14Q20, 68W30
%\end{AMS}

%We study strictly positive polynomials in the boundary of the SOS cone and extreme rays in the dual cone.
%The cases of polynomials of degree 6 in 3 variables and degree 4 in 4 variables have been completely characterized by G. Blekherman \cite{blekherman}. We focus on the cases of more variables or higher degree, for which very little is known. We show that a conjecture by Eisenbud, Green and Harris gives bounds on the ranks of the associated quadratic forms. In particular, for the case of homogeneous quadric polynomials in 5 variables, for which the conjecture has recently been proved, we obtain bounds on the maximum number of polynomials that can appear in a sum of squares decomposition, which improves the known general bounds. We provide several new examples in different cases showing that the bounds predicted by the conjectured are attainable.
\end{abstract}

\maketitle

\section{Introduction}
\label{introduction}

The decomposition of a real multivariate polynomial as a sum of squares of real polynomials is a central problem in real algebraic geometry, with many theoretical and practical applications.
An interesting and difficult problem is to determine the minimum number $p(n,2d)$ such that any polynomial of degree $2d$ in $n$ variables that is a sum of squares can be decomposed as a sum of $p(n,2d)$ squares. This number is called the Pythagoras number of $n$-ary forms of degree $2d$.

In \cite{scheiderer2017}, C. Scheiderer obtained lower bounds for $p(n,2d)$ for all $n \ge 3$ and $d \ge 2$, assuming a conjecture by Iarrovino and Kanev that can be verified computationally for small values of $n$ and $d$.
The lower bounds obtained in that paper are close to the known upper bounds for the Pythagoras number. A natural question is whether these lower bounds are sharp. For the case of polynomials of degree 4 in 5 variables, the lower bound given in \cite{scheiderer2017} is 7. That is, there exists a polynomial that is the sum of 7 squares and cannot be decomposed as the sum of 6 squares. In this paper, we give an explicit example that improves this bound. Our polynomial is a sum of 8 squares and cannot be decomposed as the sum of 7 squares. As far as we know, this is the first example that improves the lower bounds given in \cite{scheiderer2017} for any $n$ and $d$, showing that the examples constructed there are not always optimal.

\subsection{Preliminaries}
\label{section:prelim}
We set some notations and recall basic results that will be used in our construction. We refer the readers to \cite[Chapter 3 and 4]{bpt} and \cite{vinzant} for details and proofs.

\subsubsection{Sums of squares} Let $H_{n,d}$ be the space of homogeneous polynomials of degree $d$ in $n$ variables and let $\Sigma_{n,2d} \subset H_{n,2d}$ be the set of polynomials that can be decomposed as a sum of squares of polynomials, which we call shortly SOS polynomials. The set $\Sigma_{n,2d}$ is a full dimensional cone in $H_{n,2d}$, included in $P_{n,2d}$, the cone of non-negative polynomials in $H_{n,2d}$. If $f \in \Sigma_{n,2d}$, $f = p_1^2 + \dots + p_s^2$, then $p_i \in H_{n,d}$ for all $1 \le i \le s$ (that is, all the polynomials in the SOS decomposition of a homogeneous polynomial of degree $2d$ are homogeneous polynomials of degree $d$).

\subsubsection{The Gram Spectrahedron} Let $m$ be the vector of monomials of degree $d$ in $n$ variables under some monomial ordering and let $N$ be the length of this vector ($N = \dim H_{n,d}$). A homogeneous polynomial $f \in \R[x_1, \dots, x_n]$ of degree $2d$ is a sum of squares if and only if there exists a positive semidefinite matrix $A \in \R^{N \times N}$ such that
$$f = m^T A m.$$
The space of all such matrices is called the Gram spectrahedron of $f$. It is a compact convex set in the space of matrices.

\subsubsection{The dual cone and bilinear forms}
\label{section:bilineal}
The dual cone $K^*$ of a convex cone $K$  in a real vector space $V$ is the set of all linear functionals in the dual space $V^*$ that are nonnegative on $K$: $K^*=\{\ell\in V^*\,:\,\ell(x)\ge 0,\, \forall x\in K\}$.
Given a form $\ell \in H_{n,2d}^*$, we can define a bilinear form $Q_{\ell}: H_{n,d} \times H_{n,d} \rightarrow \R$, $Q_{\ell}(p,q) = \ell(pq)$. For computations, it is convenient to fix a monomial base $\mathcal{M}$ of $H_{n,d}$ and represent $Q_{\ell}$ in the coordinates of $\mathcal{M}$ as a matrix $\Qmat \in \R^{N \times N}$, where $N = \dim H_{n,d}$. If $\ell \in \Sigma_{n,2d}^*$, then $Q_{\ell}(p,p) \ge 0$ for all $p \in H_{n,d}$. That is, the form $Q_{\ell}$ is positive semi-definite and so is the matrix $\Qmat$. The converse is also true: if $\ell \in H_{n,2d}^*$ and $Q_{\ell}$ is positive semidefinite, then $\ell \in \Sigma_{n,2d}^*$.

If $f$ is in the boundary of $\Sigma_{n,2d}$, there exists $\ell \in \Sigma_{n,2d}^*$ such that $\ell(f) = 0$. If $f = p_1^2 + \dots + p_s^2$ then $\ell(p_i^2) = 0$ for all $1 \le i \le s$. Therefore, $Q_{\ell}(p_i, p_i) = 0$ and $(p_i)_{\mathcal{M}}$ is in the kernel of the matrix $\Qmat$ for all $1 \le i \le s$ (since $\Qmat$ is positive semidefinite).
%An element $v \in K \subset V$ is in the boundary of $K$ if and only if there exists $\ell \in K^*$ such that $\ell(v) = 0$.

\section{Sum of 8 polynomials}

We give an example of a polynomial of degree 4 in 5 variables that is the sum of 8 squares and cannot be decomposed as the sum of 7 squares.

\subsection{Construction}
Our starting point is a strictly positive polynomial of degree 4 in 4 variables in the boundary of $\Sigma_{4,4}$.
In \cite{blekherman}, G. Blekherman provides formulas for constructing such polynomials. We use the example given in \cite[Example 4.3]{valdettaro2020} following those formulas.

\begin{example}
\label{ex:44sumOf4}
Let $p_1, p_2, p_3, p_4 \in \R[x_1, x_2, x_3, x_4]$,
\begin{align*}
p_1 &= x_1^2- x_4^2, \\
p_2 &= x_2^2- x_4^2, \\
p_3 &= x_3^2- x_4^2, \\
p_4 &= -x_1^2 - x_1 x_2 - x_1 x_3 + x_1 x_4 - x_2 x_3 + x_2 x_4 + x_3 x_4,
\end{align*}
and set $g = p_1^2+p_2^2+p_3^2+p_4^2$. Then $g$ is a strictly positive polynomial in the boundary of $\Sigma_{4,4}$. The decomposition of $g$ as a sum of squares is unique up to orthogonal transformations (see \cite[Section 3.13]{capco}).
\end{example}

We add to this example four new polynomials in a ring with one more variable.
\begin{example}
\label{ex:54sumOf8}
In $\R[x_1, \dots, x_5]$, let $p_1, p_2, p_3, p_4$ be as in Example \ref{ex:44sumOf4},
and set
$$p_5 = x_1x_5, \quad p_6 = x_2x_5, \quad p_7 = x_3 x_5, \quad p_8 = x_4 x_5.$$
Let $g = p_1^2+p_2^2+p_3^2+p_4^2+p_5^2+p_6^2+p_7^2 + p_8^2$. Then $g \in \Sigma_{5,4}$ is the sum of 8 squares and cannot be decomposed as the sum of 7 squares.
\end{example}

\begin{remark} The intuitive idea behind this example is that since the monomial $x_5^4$ is not in $g$, the monomials $x_i^2 x_5^2$, $1 \le i \le 4$, in $g$ can only be obtained from a product $(x_i x_5)^2$, hence $x_1x_5, \dots, x_4x_5$ can be thought as \emph{new} variables.
\end{remark}

We will prove the claim in the example by brute force, that is, we will suppose that there is a decomposition of $g$ as a sum of 7 squares and prove that the resulting equations on the coefficients have no real solution.
To reduce the redundancy and simplify the computations, we show first that the polynomials in the decomposition can be assumed to be in triangular shape.

\begin{lemma}
\label{lemma:triangular}
Let $\{p_1, \dots, p_s\} \subset k[x_1, \dots, x_n]$ be a set of $s$ linearly independent homogeneous polynomials of the same degree $d$ and let $g$ be a sum of squares, $g = \sum_{i=1}^t q_i^2$, where $q_i \in \langle p_1, \dots, p_s\rangle_{\R}$, $1 \le i \le t$ (each $q_i$ is a real linear combination of $p_1, \dots, p_s$). Then for some $t' \le \min(t, s)$ there exists a decomposition of $g$ as a sum of $t'$ squares $g = \sum_{i=1}^{t'}, \tilde q_i^2$ where the polynomials $\tilde q_i \in H_{n,d}$ are in triangular shape with respect to $p_1, \dots, p_s$. That is, for any $1 \le i \le t'$, $\tilde q_i \in \langle p_i, \dots, p_s \rangle_\R$.
\end{lemma}

\begin{proof}
Let $g = q_1^2 + \dots + q_{t}^2$, with $q_i = \sum_{j=1}^s a_{ij} p_j$, $a_{ij} \in \R$, $1 \le i \le t$, $1 \le j \le s$. Taking $v = \begin{pmatrix} p_1 & p_2 & \dots & p_s \end{pmatrix}^T$, we get
$$
 g  = v^T A^T A v,
$$
where $A \in \R^{t \times s}$ is the matrix with entries $a_{ij}$.

The matrix $M = A^T A \in \R^{s \times s}$ is a positive semidefinite matrix of rank at most $\min(t, s)$. Let $t'$ be the rank of $M$. Any such matrix can be decomposed as $M = B^T B$, where $B \in \R^{t' \times s}$. Now let $B = QR$ be a QR decomposition of $B$, with $Q \in \R^{t' \times t'}$ orthogonal and $R \in \R^{t' \times s}$ upper triangular. Then $B^T B = R^T R$ and the formula $g = v^T R^T R v$ gives the decomposition of $g$ in triangular shape, taking $\tilde q_i$ as the $i$th element of $R v$.
\end{proof}

\begin{proposition}
The polynomial $g$ in Example \ref{ex:54sumOf8} is a sum of 8 squares and cannot be decomposed as the sum of 7 squares.
\end{proposition}
\begin{proof}
See \cite{pitagorasWorksheet} for the computations in Maple \cite{maple}. We first prove that any polynomial in a SOS decomposition of $g$ is a linear combination of $\{p_1, \dots, p_8\}$.
We look for non-zero linear forms $\ell \in \Sigma_{5,4}^*$ that vanish in $g$. By \cite[Lemma 2.6]{blekherman} (see also   \cite[Proposition 5.2]{valdettaro2020}), $\ell$ must satisfy $\ell(p_i q) = 0$ for all $p_i$, $1 \le i \le 8$, and all $q \in H_{5,2}$. We associate to each linear form $\ell$ the bilinear form $Q_{\ell}(p, q) = \ell(pq)$ and consider the set of bilinear forms $\{Q_{\ell}: H_{5,2} \times H_{5,2}  \rightarrow \R \mid \ell \in H_{5,4}^*, \ell(p_i q) = 0 \text{ for all $p_i$, $1 \le i \le 8$, and all $q \in H_{5,2}$}\}$. For the monomial base
 $$\mathcal{M} = \{x_1^2, x_1x_2, x_1x_3, x_1x_4, x_2^2, x_2x_3, x_2x_4, x_3^2, x_3x_4, x_4^2, x_1x_5, x_2x_5, x_3x_5, x_4x_5, x_5^2\}$$
 we compute the space of matrices $\Qmat$ corresponding to those bilinear forms . We obtain a 2-dimensional space $E = \{\Qmat(t_1, t_2) : t_1, t_2 \in \R\}$, where
%$$\left(
%\begin{array}{r*{15}{l}}
% 6t_1 & -t_1 & -t_1 & t_1 & 0 & 6t_1 & -t_1 & t_1 & 0 & 6t_1 & t_1 & 0 & 6t_1 & 0 & 0\\
% -t_1 & 6t_1 & -t_1 & t_1 & 0 & -t_1 & -t_1 & t_1 & 0 & -t_1 & t_1 & 0 & -t_1 & 0 & 0\\
% -t_1 & -t_1 & 6t_1 & t_1 & 0 & -t_1 & -t_1 & t_1 & 0 & -t_1 & t_1 & 0 & -t_1 & 0 & 0\\
%  t_1 & t_1 & t_1 & 6t_1 & 0 & t_1 & t_1 & -t_1 & 0 & t_1 & -t_1 & 0 & t_1 & 0 & 0\\
%       0 & 0 & 0 & 0 & 0 & 0 & 0 & 0 & 0 & 0 & 0 & 0 & 0 & 0 & 0\\
%6t_1 & -t_1 & -t_1 & t_1 & 0 & 6t_1 & -t_1 & t_1 & 0 & 6t_1 & t_1 & 0 & 6t_1 & 0 & 0\\
% -t_1 & -t_1 & -t_1 & t_1 & 0 & -t_1 & 6t_1 & t_1 & 0 & -t_1 & t_1 & 0 & -t_1 & 0 & 0\\
%  t_1 & t_1 & t_1 & -t_1 & 0 & t_1 & t_1 & 6t_1 & 0 & t_1 & -t_1 & 0 & t_1 & 0 & 0\\
%       0 & 0 & 0 & 0 & 0 & 0 & 0 & 0 & 0 & 0 & 0 & 0 & 0 & 0 & 0\\
%6t_1 & -t_1 & -t_1 & t_1 & 0 & 6t_1 & -t_1 & t_1 & 0 & 6t_1 & t_1 & 0 & 6t_1 & 0 & 0\\
%t_1 & t_1 & t_1 & -t_1 & 0 & t_1 & t_1 & -t_1 & 0 & t_1 & 6t_1 & 0 & t_1 & 0 & 0\\
%0 & 0 & 0 & 0 & 0 & 0 & 0 & 0 & 0 & 0 & 0 & 0 & 0 & 0 & 0\\
%6t_1 & -t_1 & -t_1 & t_1 & 0 & 6t_1 & -t_1 & t_1 & 0 & 6t_1 & t_1 & 0 & 6t_1 & 0 & 0\\
%0 & 0 & 0 & 0 & 0 & 0 & 0 & 0 & 0 & 0 & 0 & 0 & 0 & 0 & 0\\
%0 & 0 & 0 & 0 & 0 & 0 & 0 & 0 & 0 & 0 & 0 & 0 & 0 & 0 & t_2
%\end{array}
%\right)
%$$
$\Qmat(t_1, t_2) = \left(\begin{array}{c|c|c}
t_1 A & & \\ \hline
 & \bm{0}_{4 \times 4}  & \\ \hline
 & & t_2
\end{array}\right) \in \R^{15 \times 15}
$ with
%\[
%\left(\begin{array}{c|c|c}
%A \in \R^{10 \times 10} & & \\ \hline
% & 0 \in \R^{4 \times 4} & \\ \hline
% & & B \in \R^{1 \times 1}
%\end{array}\right)
%\]
{\small
$$A = \begin{pmatrix}
6 & -1 & -1 & 1 & 6 & -1 & 1 & 6 & 1 & 6 \\
-1 & 6 & -1 & 1 & -1 & -1 & 1 & -1 & 1 & -1 \\
-1 & -1 & 6 & 1 & -1 & -1 & 1 & -1 & 1 & -1 \\
1 & 1 & 1 & 6 & 1 & 1 & -1 & 1 & -1 & 1 \\
6 & -1 & -1 & 1 & 6 & -1 & 1 & 6 & 1 & 6 \\
-1 & -1 & -1 & 1 & -1 & 6 & 1 & -1 & 1 & -1 \\
1 & 1 & 1 & -1 & 1 & 1 & 6 & 1 & -1 & 1 \\
6 & -1 & -1 & 1 & 6 & -1 & 1 & 6 & 1 & 6 \\
1 & 1 & 1 & -1 & 1 & 1 & -1 & 1 & 6 & 1 \\
6 & -1 & -1 & 1 & 6 & -1 & 1 & 6 & 1 & 6
\end{pmatrix}.
$$
}

A linear form $\ell \in H_{5,4}^*$ is in $\ell \in \Sigma_{5,4}^*$ iff the bilinear form $Q_{\ell}$  is positive semidefinite.
The matrix $\Qmat(t_1, t_2)$ is a block matrix, with a block depending on $t_1$ and a block depending on $t_2$. The matrix $A$ is positive semidefinite, so for any positive numbers $t_1$ and $t_2$ the matrix $\Qmat(t_1, t_2)$ is positive semidefinite. Setting $\Qmat = \Qmat(1,1)$, the resulting matrix has kernel of dimension 8:
$$
W = \langle p_1, \dots, p_8 \rangle_\R,
$$
which proves our first claim, because any polynomial in a SOS decomposition of $g$ must be in the kernel of $\Qmat$ (see Section \ref{section:bilineal}).

Now we prove that there is no decomposition of $g$ as a sum of 7 squares. By the first part and Lemma \ref{lemma:triangular}, we can assume that the polynomials in the decomposition are in triangular shape with respect to $p_1, \dots, p_8$. That is, we can assume
$$\arraycolsep=.4pt
\begin{array}{r*{18}{l}}
f_1 &= u_{11} p_1 &+ &u_{12} p_2 &+ &u_{13} p_3 &+ &u_{14} p_4 &+ &u_{15} p_5 &+ &u_{16} p_6 &+ &u_{17} p_7 &+ &u_{18} p_8, \\
f_2 &=            &  &u_{22} p_2 &+ &u_{23} p_3 &+ &u_{24} p_4 &+ &u_{25} p_5 &+ &u_{26} p_6 &+ &u_{27} p_7 &+ &u_{28} p_8, \\
f_3 &=            &  &           &  &u_{33} p_3 &+ &u_{34} p_4 &+ &u_{35} p_5 &+ &u_{36} p_6 &+ &u_{37} p_7 &+ &u_{38} p_8, \\
f_4 &=            &  &           &  &           &  &u_{44} p_4 &+ &u_{45} p_5 &+ &u_{46} p_6 &+ &u_{47} p_7 &+ &u_{48} p_8, \\
f_5 &=            &  &           &  &           &  &           &  &u_{55} p_5 &+ &u_{56} p_6 &+ &u_{57} p_7 &+ &u_{58} p_8, \\
f_6 &=            &  &           &  &           &  &           &  &           &  &u_{66} p_6 &+ &u_{67} p_7 &+ &u_{68} p_8, \\
f_7 &=            &  &           &  &           &  &           &  &           &  &           &  &u_{77} p_7 &+ &u_{78} p_8.
\end{array}
$$

The equation $g = f_1^2 + \dots + f_7^2$ defines a system of quadratic equations in the coefficients $u_{ij}$. To solve this system, we compute a Groebner basis of the ideal defined by the equations using the degree reverse lexicographical ordering. We obtain the basis $\{1\}$, which proves that there is no solution to the system of equations. That is, $g$ cannot be decomposed as a sum of 7 squares.
\end{proof}

\begin{remark} In the first part of the above proof, we have shown that the polynomial $g$ in Example \ref{ex:54sumOf8} is a sum of at most 8 linearly independent squares. This means that the Gram spectrahedron of $g$ contains points of rank at most $8$. In the second part we have shown that the Gram spectrahedron of $g$ contains no point of rank $7$. Both results combined imply that the Gram spectrahedron consists of only one point (or otherwise there would be points of different rank). That is, the decomposition of $g$ given in Example \ref{ex:54sumOf8} is the unique (up to orthogonal equivalence) decomposition of $g$ as a sum of squares.

This can also be verified computationally adding an eighth polynomial $f_8 = u_{88} x_4 x_5$ in the above proof and computing the Groebner basis of the resulting system of equations (see the auxiliary code \cite{pitagorasWorksheet} for this computation).
\end{remark}

\begin{remark}
If we follow this strategy to construct examples for other values of $n$, we would start with a sequence of parameters of degree 2 in $n-1$ variables, hence consisting of $n-1$ polynomials, and add $n-1$ more polynomials $x_1 x_{n}, x_2x_n, \dots, x_{n-1} x_{n}$. The total number of polynomials is then $2(n-1)$. If it also holds that the resulting polynomial cannot be decomposed as a sum of less polynomials, this bound would improve the bound given in \cite{scheiderer2017} only in the case $(n,d) = (5,4)$, which is the case studied in this paper, so we don't attempt to extend our construction to larger values of $n$.
\end{remark}

%\section*{Acknowledgments}
%The authors would like to thank Jose Capco, Gabriela Jeronimo, Daniel Perrucci and Claus Scheiderer for many fruitful discussions. Santiago Laplagne would like to thank Claus Scheiderer and all the crew of Konstanz University for their kind hospitality during his visit to Konstanz.

\bibliographystyle{amsplain}
\bibliography{rationalSOS}

\end{document}